\documentclass[graybox]{svmult}

\usepackage{mathptmx}       
\usepackage{helvet}         
\usepackage{courier}        
\usepackage{type1cm}        
\usepackage{makeidx}         
\usepackage{graphicx}        
\usepackage{multicol}        
\usepackage[bottom]{footmisc}
\usepackage{amsmath}
\usepackage{amssymb}

\usepackage{algorithm}
\usepackage{algorithmic}

\makeindex             

\newcommand{\R}{{\mathbb R}}
\newcommand{\Ex}{{\mathbb E}}

\newcommand{\Prb}{{\mathbb P}}

\newcommand{\e}{\varepsilon}
\newcommand{\la}{\langle}
\newcommand{\ra}{\rangle}

\begin{document}

\title*{Mirror Descent and Convex Optimization Problems With Non-Smooth Inequality Constraints}
\titlerunning{Mirror Descent for Convex Constrained Problems}
\author{Anastasia Bayandina, Pavel Dvurechensky, Alexander Gasnikov, Fedor Stonyakin, and Alexander Titov}
\authorrunning{A. Bayandina, P. Dvurechensky, A. Gasnikov, F. Stonyakin, A. Titov}
\institute{Anastasia Bayandina \at Moscow Institute of Physics and Technology, 9 Institutskiy per., Dolgoprudny, Moscow Region, 141701, Russia and Skolkovo Institute of Science and Technology, Skolkovo Innovation Center, Building 3, Moscow,  143026, Russia \email{anast.bayandina@gmail.com}
\and Pavel Dvurechensky \at Weierstrass Institute for Applied Analysis and Stochastics, Mohrenstr. 39, Berlin, 10117, Germany and Institute for Information Transmission Problems RAS, Bolshoy Karetny per. 19, build.1, Moscow, 127051, Russia \email{pavel.dvurechensky@wias-berlin.de}
\and  Alexander Gasnikov \at Moscow Institute of Physics and Technology, 9 Institutskiy per., Dolgoprudny, Moscow Region, 141701, Russia \email{gasnikov@yandex.ru}
\and Fedor Stonyakin \at V.I. Vernadsky Crimean Federal University, 4 V. Vernadsky Ave, Simferopol, \email{fedyor@mail.ru}
\and Alexander Titov \at Moscow Institute of Physics and Technology, 9 Institutskiy per., Dolgoprudny, Moscow Region, 141701, Russia \email{a.a.titov@phystech.edu}}
%
%
\maketitle

\abstract*{
We consider the problem of minimization of a convex function on a simple set with convex non-smooth inequality constraint and describe first-order methods to solve such problems in different situations: smooth or non-smooth objective function; convex or strongly convex objective and constraint; deterministic or randomized information about the objective and constraint. Described methods are based on Mirror Descent algorithm and switching subgradient scheme. One of our focus is to propose, for the listed different settings, a Mirror Descent with adaptive stepsizes and adaptive stopping rule. We also construct Mirror Descent for problems with objective function, which is not Lipschitz, e.g. is a quadratic function. Besides that, we address the question of recovering the dual solution in the considered problem.
}

\abstract{
We consider the problem of minimization of a convex function on a simple set with convex non-smooth inequality constraint and describe first-order methods to solve such problems in different situations: smooth or non-smooth objective function; convex or strongly convex objective and constraint; deterministic or randomized information about the objective and constraint. Described methods are based on Mirror Descent algorithm and switching subgradient scheme. One of our focus is to propose, for the listed different settings, a Mirror Descent with adaptive stepsizes and adaptive stopping rule. We also construct Mirror Descent for problems with objective function, which is not Lipschitz, e.g. is a quadratic function. Besides that, we address the question of recovering the dual solution in the considered problem.
}

\section{Introduction}
\label{S:Intro}

We consider the problem of minimization of a convex function on a simple set with convex non-smooth inequality constraint and describe first-order methods to solve such problems in different situations: smooth or non-smooth objective function; convex or strongly convex objective and constraint; deterministic or randomized information about the objective and constraint. The reason for considering first-order methods is potential large (more than $10^5$) number of decision variables.

Because of the non-smoothness presented in the problem, we consider subgradient methods.
These methods have a long history starting with the method for deterministic unconstrained problems and Euclidean setting in \cite{shor1967generalized} and the generalization for constrained problems in \cite{polyak1967general}, where the idea of steps switching between the direction of subgradient of the objective and the direction of subgradient of the constraint was suggested. Non-Euclidean extension, usually referred to as Mirror Descent, originated in \cite{nemirovskii1979efficient,nemirovsky1983problem} and later analyzed in \cite{beck2003mirror}.
An extension for constrained problems was proposed in \cite{nemirovsky1983problem}, see also recent version in \cite{beck2010comirror}.
Mirror Descent for unconstrained stochastic optimization problems was introduced in \cite{nemirovski2009robust}, see also \cite{juditsky2012first-order,nedic2014stochastic}, and extended for stochastic optimization problems with expectation constraints in \cite{lan2016algorithms}.
To prove faster convergence rate of Mirror Descent for strongly convex objective in unconstrained case, the restart technique \cite{nemirovskii1985optimal,nemirovsky1983problem,nesterov1983method} was used in \cite{juditsky2012first-order}. An alternative approach for strongly convex stochastic optimization problems with strongly convex expectation constraints is used in \cite{lan2016algorithms}.

Usually, the stepsize and stopping rule for Mirror Descent requires to know the Lipschitz constant of the objective function and constraint, if any.
Adaptive stepsizes, which do not require this information, are considered in \cite{ben-tal2001lectures} for problems without inequality constraints, and in \cite{beck2010comirror} for constrained problems. Nevertheless, the stopping criterion, expressed in the number of steps, still requires knowledge of Lipschitz constants.
One of our focus in this chapter is to propose, for constrained problems, a Mirror Descent with adaptive stepsizes and adaptive stopping rule. We also adopt the ideas of \cite{nesterov2004introduction,nesterov2015subgradient} to construct Mirror Descent for problems with objective function, which is not Lipschitz, e.g. a quadratic function.
Another important issue, we address, is recovering the dual solution of the considered problem, which was considered in different contexts in \cite{anikin2016randomization,bayandina2018primal-dual,nesterov2015new}.

Formally speaking, we consider the following convex constrained minimization problem
\begin{align}
\label{eq:PrSt}
    \min \{ f(x) : \quad x \in X \subset E, \quad g(x) \leq 0\},
\end{align}
where $X$ is a convex closed subset of a finite-dimensional real vector space $E$, $f: X \to \R$, $g: E \to \R$ are convex functions.

We assume $g$ to be a non-smooth Lipschitz-continuous function and the problem \eqref{Alg:MDNS} to be regular. The last means that there exists a point $\bar{x}$ in relative interior of the set $X$, such that $g(\bar{x}) < 0$.

Note that, despite problem \eqref{eq:PrSt} contains only one inequality constraint, considered algorithms allow to solve more general problems with a number of constraints given as $\{ g_i(x) \leq 0, i=1,...,m \}$. The reason is that these constraints can be aggregated and represented as an equivalent constraint given by $\{ g(x) \leq 0\}$, where $g(x)=\max_{i=1,...,m}g_i(x)$.

The the rest of the chapter is divided in three parts. In Section~\ref{S:MD}, we describe some basic facts about Mirror Descent, namely, we define the notion of proximal setup, the Mirror Descent step, and provide the main lemma about the progress on each iteration of this method. Section~\ref{S:Det} is devoted to deterministic constrained problems, among which we consider convex non-smooth problems, strongly convex non-smooth problems and convex problems with smooth objective. The last, Section~\ref{S:Rand}, considers randomized setting with available stochastic subgradients for the objective and constraint and possibility to calculate the constraint function. We consider methods for convex and strongly convex problems and provide complexity guarantees in terms of expectation of the objective residual and constraint infeasibility, as long as in terms of large deviation probability for these two quantities.

\textbf{Notation:} Given a subset $I$ of natural numbers, we denote $|I|$ the number of its elements.

\section{Mirror Descent Basics}
\label{S:MD}
We consider algorithms, which are based on Mirror Descent method. Thus, we start with the description of proximal setup and basic properties of Mirror Descent step.
Let $E$ be a finite-dimensional real vector space and $E^*$ be its dual. We denote the value of a linear function $g \in E^*$ at $x\in E$ by $\la g, x \ra$. Let $\|\cdot\|_{E}$ be some norm on $E$, $\|\cdot\|_{E,*}$ be its dual, defined by $\|g\|_{E,*} = \max\limits_{x} \big\{ \la g, x \ra, \| x \|_E \leq 1 \big\}$. We use $\nabla f(x)$ to denote any subgradient of a function $f$ at a point $x \in {\rm dom} f$.

We choose a {\it prox-function} $d(x)$, which is continuous, convex on $X$ and
\begin{enumerate}
	\item admits a continuous in $x \in X^0$ selection of subgradients 	$\nabla d(x)$, where $X^0 \subseteq X$  is the set of all $x$, where $\nabla d(x)$ exists;
	\item $d(x)$ is $1$-strongly convex on $X$ with respect to $\|\cdot\|_{E}$, i.e., for any $x \in X^0, y \in X$ $d(y)-d(x) -\la \nabla d(x) ,y-x \ra \geq \frac12\|y-x\|_{E}^2$.
\end{enumerate}
Without loss of generality, we assume that $\min\limits_{x\in X} d(x) = 0$.

We define also the corresponding {\it Bregman divergence} $V[z] (x) = d(x) - d(z) - \la \nabla d(z), x - z \ra$, $x \in X, z \in X^0$. Standard proximal setups, i.e. Euclidean, entropy, $\ell_1/\ell_2$, simplex, nuclear norm, spectahedron can be found in \cite{ben-tal2015lectures}.

Given a vector $x\in X^0$, and a vector $p \in E^*$, the Mirror Descent step is defined as

\begin{equation}
x_+ = \mathrm{Mirr}[x](p) := \arg\min\limits_{u\in X} \big\{ \la p, u \ra + V[x](u) \big\} = \arg\min\limits_{u\in X} \big\{ \la p, u \ra + d(u) - \la \nabla d(x) , u \ra  \big\}.
\label{eq:MDStep}
\end{equation}
We make the simplicity assumption, which means that $\mathrm{Mirr}[x] (p)$ is easily computable.
The following lemma \cite{ben-tal2001lectures} describes the main property of the Mirror Descent step. We prove it here for the reader convenience and to make the chapter self-contained.
\begin{lemma}
	\label{Lm:MDProp}
    Let $f$ be some convex function over a set $X$, $h > 0$ be a stepsize, $x \in X^{0}$. Let the point $x_+$ be defined by
   $ x_+ = \mathrm{Mirr}[x](h \cdot (\nabla f(x) + \Delta)) $, where $\Delta \in E^*$. Then, for any $u \in X$,
    \begin{align}
    h \cdot \big( f(x) - f(u) + \la \Delta , x - u \ra \big) & \leq h \cdot \la \nabla f(x) + \Delta , x - u \ra \notag \\
		& \leq \frac{h^2}{2} \| \nabla f (x)  + \Delta \|^2_{E,*} + V[x](u) - V[x_+](u).				 \label{eq:MDProp}
    \end{align}
\end{lemma}
\begin{proof}
\smartqed
By optimality condition in \eqref{eq:MDStep}, we have that there exists a subgradient $\nabla d(x_+)$, such that, for all $u \in X$,
$$
\la h \cdot (\nabla f(x) + \Delta) + \nabla d(x_+) - \nabla d(x) , u - x_+ \ra \geq 0.
$$
Hence, for all $u \in X$,
\begin{align}
\la h \cdot (\nabla f(x) + \Delta) , x- u \ra & \leq \la h \cdot (\nabla f(x) + \Delta) , x- x_+ \ra + \la \nabla d(x_+)-\nabla d(x),u - x_+ \ra \notag \\
& = \la h \cdot (\nabla f(x) + \Delta) , x- x_+ \ra  + \left(d(u) - d(x) - \la \nabla d(x), u-x \ra \right) \notag \\
& \hspace{1em} - \left(d(u) - d(x_+) - \la \nabla d(x_+), u-x_+ \ra \right) \notag \\
& \hspace{1em} - \left(d(x_+) - d(x) - \la \nabla d(x), x_+-x \ra \right) \notag \\
& \leq \la h \cdot (\nabla f(x) + \Delta) , x- x_+ \ra  + V[x](u)-V[x_+](u) - \frac{1}{2}\|x_+-x\|_E^2 \notag \\
& \leq V[x](u)-V[x_+](u) + \frac{h^2}{2}\|(\nabla f(x) + \Delta)\|_{E,*}^2,\notag
\end{align}
where we used the fact that, for any $g \in E^*$,
$$
\max_{y\in E} \la g, y \ra - \frac{1}{2}\|y\|_E^2 = \frac{1}{2}\|g\|_{E,*}^2.
$$
By convexity of $f$, we obtain the left inequality in \eqref{eq:MDProp}.
\qed
\end{proof}

\section{Deterministic Constrained Problems}
\label{S:Det}
In this section, we consider problem \eqref{eq:PrSt} in two different settings, namely, non-smooth Lipschitz-continuous objective function $f$ and general objective function $f$, which is not necessarily Lipschitz-continuous, e.g. a quadratic function. In both cases, we assume that $g$ is non-smooth and is Lipschitz-continuous
\begin{equation}
|g(x)-g(y)|\leq M_g\|x-y\|_E, \quad x,y \in X.
\label{eq:gLipCont}
\end{equation}
Let $x_*$ be a solution to \eqref{eq:PrSt}. We say that a point $\tilde{x} \in X$ is an \textit{$\e$-solution} to \eqref{eq:PrSt} if
\begin{equation}
    \label{eq:DetSolDef}
        f(\tilde{x}) - f(x_{*}) \leq \e, \quad g(\tilde{x}) \leq \e.
\end{equation}

The methods we describe are based on the of Polyak's switching subgradient method \cite{polyak1967general} for constrained convex problems, also analyzed in \cite{nesterov2004introduction}, and Mirror Descent method originated in \cite{nemirovsky1983problem}; see also \cite{ben-tal2001lectures}.

\subsection{Convex Non-Smooth Objective Function}
\label{S:CNSOF}
In this subsection, we assume that $f$ is a non-smooth Lipschitz-continuous function
\begin{equation}
|f(x)-f(y)|\leq M_f\|x-y\|_E, \quad x,y \in X.
\label{eq:fLipCont}
\end{equation}
Let $x_*$ be a solution to \eqref{eq:PrSt} and assume that we know a constant $\Theta_0 > 0$ such that
\begin{equation}
d(x_{*}) \leq \Theta_0^2.
\label{eq:dx*Bound}
\end{equation}
For example, if $X$ is a compact set, one can choose $\Theta_0^2 = \max_{ x \in X} d(x)$.
We further develop line of research \cite{anikin2016randomization,bayandina2018primal-dual}, but we should also mention close works \cite{beck2010comirror,nesterov2015new}. In comparison to known algorithms in the literature, the main advantage of our method for solving \eqref{eq:PrSt}  is that the stopping criterion does not require the knowledge of constants $M_f,M_g$, and, in this sense, the method is adaptive. Mirror Descent with stepsizes not requiring knowledge of Lipschitz constants can be found, e.g., in \cite{ben-tal2001lectures} for problems without inequality constraints, and, for constrained problems, in \cite{beck2010comirror}.The algorithm is similar to the one in  \cite{bayandina2017adaptive}, but, for the sake of consistency with other parts of the chapter, we use slightly different proof.

\begin{algorithm}[h!]
\caption{Adaptive Mirror Descent (Non-Smooth Objective)}
\label{Alg:MDNS}
\begin{algorithmic}[1]
  \REQUIRE accuracy $\e > 0$; $\Theta_0$ s.t. $d(x_{*}) \leq \Theta_0^2$.
  \STATE  $x^0 = \arg\min\limits_{x\in X} d(x)$.
	\STATE Initialize the set $I$ as empty set.
	\STATE Set $k=0$.
	\REPEAT
    \IF{$g(x^k) \leq \e$}
      \STATE $M_k = \| \nabla f(x^k) \|_{E,*}$,
      \STATE $h_k = \frac{\e}{M_k^2}$
      \STATE $x^{k+1} =\mathrm{Mirr}[x^k](h_k \nabla f(x^k))$ ("productive step")
      \STATE Add $k$ to $I$.
    \ELSE
      \STATE $M_k = \| \nabla g(x^k) \|_{E,*}$
      \STATE $h_k = \frac{\e}{M_k^2}$
      \STATE $x^{k+1} = \mathrm{Mirr}[x^k](h_k \nabla g(x^k))$ ("non-productive step")
    \ENDIF
    \STATE Set $k = k + 1$.
  \UNTIL{$\sum\limits_{j =0 }^{k-1} \frac{1}{M_j^2} \geq \frac{2\Theta_0^2}{\e^2}$}
  \ENSURE $\bar{x}^k := \frac{\sum\limits_{i\in I} h_i x^i}{\sum\limits_{i\in I} h_i}$
\end{algorithmic}
\end{algorithm}

\begin{theorem}
	\label{Th:MDCompl}
	Assume that inequalities \eqref{eq:gLipCont} and \eqref{eq:fLipCont} hold and a known constant $\Theta_0 > 0$ is such that $
d(x_{*}) \leq \Theta_0^2$. Then, Algorithm \ref{Alg:MDNS} stops after not more than
	\begin{equation}
	k = \left\lceil\frac{2\max\{M_f^2,M_g^2\} \Theta_0^2}{\e^2}\right\rceil
	\label{eq:MDNSComplEst}
	\end{equation}
	iterations and $\bar{x}^k$ is an $\e$-solution to \eqref{eq:PrSt} in the sense of \eqref{eq:DetSolDef}.
\end{theorem}
\begin{proof}
\smartqed
First, let us prove that the inequality in the stopping criterion holds for $k$ defined in \eqref{eq:MDNSComplEst}.
By \eqref{eq:gLipCont} and \eqref{eq:fLipCont}, we have that, for any $i \in \{0,...,k-1\}$, $M_i \leq \max\{M_f,M_g\}$. Hence, by \eqref{eq:MDNSComplEst}, $\sum\limits_{j =0 }^{k-1} \frac{1}{M_j^2} \geq \frac{k}{\max\{M_f^2,M_g^2\}} \geq \frac{2\Theta_0^2}{\e^2}$.

Denote $[k] = \{ i \in \{0,...,k-1\} \}$, $J = [k] \setminus I$. From Lemma \ref{Lm:MDProp} with $\Delta = 0$, we have, for all $i \in I$ and all $u \in X$,
$$
h_i \cdot \big( f(x^i) - f(u) \big) \leq \frac{h_i^2}{2} \| \nabla f (x^i) \|^2_{E,*} + V[x^i](u) - V[x^{i+1}](u)
$$
and, for all $i \in J$ and all $u \in X$,
$$
h_i \cdot \big( g(x^i) - g(u) \big) \leq \frac{h_i^2}{2} \| \nabla g (x^i) \|^2_{E,*} + V[x^i](u) - V[x^{i+1}](u).
$$
Summing up these inequalities for $i$ from $0$ to $k-1$, using the definition of $h_i$, $i \in \{0,...,k-1\}$, and taking $u=x_*$, we obtain
\begin{align}
  &\sum\limits_{i\in I} h_i \big( f(x^i) - f(x_{*}) \big) + \sum\limits_{i\in J} h_i \big( g(x^i) - g(x_{*}) \big) \notag \\
  & \leq \sum\limits_{i\in I} \frac{h_i^2 M_i^2}{2} + \sum\limits_{i\in J} \frac{h_i^2 M_i^2}{2} + \sum\limits_{i\in [k]} \big( V[x^i](x_{*}) - V[x^{i+1}](x_{*}) \big)\notag \\
	& \leq \frac{\e}{2} \sum\limits_{i\in [k]} h_i + \Theta_0^2.
	\label{eq:ThMDCproof1}
\end{align}
We also used that, by definition of $x^0$ and \eqref{eq:dx*Bound},
$$
V[x^0](x_{*}) = d(x_{*}) - d(x^0) - \la \nabla d(x^0), x_{*} - x^0 \ra \leq  d(x_{*}) \leq \Theta_0^2.
$$
Since, for $i\in J$, $g(x^i) - g(x_{*}) \geq g(x^i) > \e$, by convexity of $f$ and the definition of $\bar{x}^k$, we have
\begin{align}
\left(\sum\limits_{i\in I} h_i \right) \big( f(\bar{x}^k) - f(x_{*}) \big) & \leq \sum\limits_{i\in I} h_i \left( f(x^i) - f(x_{*}) \right) < \frac{\e}{2} \sum\limits_{i\in [k]} h_i - \e\sum\limits_{i\in J} h_i + \Theta_0^2 \notag \\
& = \e\sum\limits_{i\in I} h_i - \frac{\e^2}{2} \sum\limits_{i\in [k]} \frac{1}{M_i^2} + \Theta_0^2 \leq \e\sum\limits_{i\in I} h_i,
\label{eq:ThMDCproof2}
\end{align}
where in the last inequality, the stopping criterion is used. As long as the inequality is strict, the case of the empty $I$ is impossible. Thus, the point $\bar{x}^k$ is correctly defined. Dividing both parts of the inequality by $\sum\limits_{i\in I} h_i$, we obtain the left inequality in \eqref{eq:DetSolDef}.

For $i\in I$, it holds that $g(x^i) \leq \e$. Then, by the definition of $\bar{x}^k$ and the convexity of $g$,
$$
  g(\bar{x}^k) \leq \left(\sum\limits_{i\in I} h_i\right)^{-1} \sum\limits_{i\in I} h_i g(x^i) \leq \e.
$$
\qed
\end{proof}

Let us now show that Algorithm \ref{Alg:MDNS} allows to reconstruct an approximate solution to the problem, which is dual to \eqref{eq:PrSt}. We consider a special type of problem \eqref{eq:PrSt} with $g$ given by
\begin{equation}
g(x) = \max\limits_{i \in \{1,..., m\}} \big\{ g_i(x) \big\}.
\label{eq:gMaxDef}
\end{equation}
Then, the dual problem to \eqref{eq:PrSt} is
\begin{equation}\label{eq:dualPrSt}
    \varphi (\lambda) = \min\limits_{x\in X} \Big\{ f(x) + \sum\limits_{i=1}^{m} \lambda_i g_i(x) \Big\} \rightarrow \max\limits_{\lambda_i \geq 0, i = 1, ..., m}\varphi (\lambda),
\end{equation}
where $\lambda_i \geq 0, i = 1, ..., m$ are Lagrange multipliers.

We slightly modify the assumption \eqref{eq:dx*Bound} and assume that the set $X$ is bounded and that we know a constant $\Theta_0 > 0$ such that
$$
\max_{x \in X} d(x) \leq \Theta_0^2.
$$

As before, denote $[k] = \{ j \in \{0,...,k-1\} \}$, $J = [k] \setminus I$.
Let $j \in J$. Then a subgradient of $g(x)$ is used to make the $j$-th step of Algorithm \ref{Alg:MDNS}. To find this subgradient, it is natural to find an active constraint $i \in 1, ..., m$ such that $g(x^j) = g_{i}(x^j)$ and use $\nabla g(x^j) = \nabla g_{i}(x^j)$ to make a step. Denote $i(j) \in 1, ..., m$ the number of active constraint, whose subgradient is used to make a non-productive step at iteration $j \in J$. In other words, $g(x^j) = g_{i(j)}(x^j)$ and $\nabla g(x^j) = \nabla g_{i(j)}(x^j)$.
We define an approximate dual solution on a step $k\geq 0$ as
    \begin{equation}\label{eq:lamDef}
        \bar{\lambda}_i^k = \frac{1}{\sum\limits_{j\in I} h_j} \sum\limits_{j\in J, i(j) = i} h_j, \quad i \in \{1, ..., m\}.
    \end{equation}
		and modify Algorithm \ref{Alg:MDNS} to return a pair $(\bar{x}^k,\bar{\lambda}^k)$.

\begin{theorem}
\label{Th:MDPDCompl}
    Assume that the set $X$ is bounded, the inequalities \eqref{eq:gLipCont} and \eqref{eq:fLipCont} hold and a known constant $\Theta_0 > 0$ is such that $
d(x_{*}) \leq \Theta_0^2$.
    Then, modified Algorithm \ref{Alg:MDNS} stops after not more than
	\begin{equation}
	k = \left\lceil\frac{2\max\{M_f^2,M_g^2\} \Theta_0^2}{\e^2}\right\rceil \notag
		\end{equation}
	iterations and the pair $(\bar{x}^k,\bar{\lambda}^k)$ returned by this algorithm satisfies
    \begin{equation}
		\label{eq:DetPDSolDef}
        f(\bar{x}^k) - \varphi(\bar{\lambda}^k) \leq \e, \quad g(\bar{x}^k) \leq \e.
    \end{equation}
\end{theorem}
\begin{proof}
\smartqed
From Lemma \ref{Lm:MDProp} with $\Delta = 0$, we have, for all $j \in I$ and all $u \in X$,
$$
h_j \big( f(x^j) - f(u) \big) \leq \frac{h_j^2}{2} \| \nabla f (x^j) \|^2_{E,*} + V[x^j](u) - V[x^{j+1}](u)
$$
and, for all $j \in J$ and all $u \in X$,
\begin{align}
h_j \big( g_{i(j)}(x^j) - g_{i(j)}(u) \big) & \leq h_j \la \nabla g_{i(j)}(x^j), x^j - u \ra \notag \\
& =  h_j \la \nabla g(x^j), x^j - u \ra \notag \\
& \leq \frac{h_j^2}{2} \| \nabla g (x^j) \|^2_{E,*} + V[x^j](u) - V[x^{j+1}](u). \notag
\end{align}
Summing up these inequalities for $j$ from $0$ to $k-1$, using the definition of $h_j$, $j \in \{0,...,k-1\}$, we obtain, for all $u \in X$,
\begin{align}
  \sum\limits_{j\in I} h_j \big( f(x^j) - f(u) \big) &+ \sum\limits_{j\in J} h_j \big( g_{i(j)}(x^j) - g_{i(j)}(u) \big)\notag \\
  & \leq \sum\limits_{i\in I} \frac{h_j^2 M_j^2}{2} + \sum\limits_{j\in J} \frac{h_j^2 M_j^2}{2} + \sum\limits_{j\in [k]} \big( V[x^j](u) - V[x^{j+1}](u) \big)\notag \\
	& \leq \frac{\e}{2} \sum\limits_{j\in [k]} h_j + \Theta_0^2.  \notag
\end{align}
Since, for $j\in J$, $g_{i(j)}(x^j) = g(x^j) > \e$, by convexity of $f$ and the definition of $\bar{x}^k$, we have, for all $u \in X$,
\begin{align}
\left(\sum\limits_{j\in I} h_j \right) \big( f(\bar{x}^k) - f(u) \big) & \leq \sum\limits_{j\in I} h_j \left( f(x^j) - f(u) \right) \notag \\
& \leq \frac{\e}{2} \sum\limits_{j\in [k]} h_j + \Theta_0^2 - \sum\limits_{j\in J} h_j \big( g_{i(j)}(x^j) - g_{i(j)}(u) \big)  \notag \\
& < \frac{\e}{2} \sum\limits_{j\in [k]} h_i  + \Theta_0^2 - \e\sum\limits_{j\in J} h_i + \sum\limits_{j\in J} h_j g_{i(j)}(u) \notag \\
& = \e\sum\limits_{j\in I} h_j - \frac{\e^2}{2} \sum\limits_{j\in [k]} \frac{1}{M_j^2} + \Theta_0^2 + \sum\limits_{j\in J} h_j g_{i(j)}(u) \notag \\
& \leq \e\sum\limits_{j\in I} h_j + \sum\limits_{j\in J} h_j g_{i(j)}(u),
\label{eq:ThMDPDComplproof1}
\end{align}
where in the last inequality, the stopping criterion is used.
At the same time, by \eqref{eq:lamDef}, for all $u \in X$,
$$
\sum\limits_{j\in J} h_j g_{i(j)}(u) = \sum_{i=1}^m \sum_{j\in J, i(j)=i} h_j g_{i(j)}(u) = \left(\sum\limits_{j\in I} h_j \right) \sum_{i=1}^m \bar{\lambda}_i^k g_{i}(u).
$$
This and \eqref{eq:ThMDPDComplproof1} give, for all $u \in X$,
\begin{align}
\left(\sum\limits_{j\in I} h_j \right)  f(\bar{x}^k)  & < \left(\sum\limits_{j\in I} h_j \right) \left( f(u) + \e +  \sum_{i=1}^m \bar{\lambda}_i^k g_{i}(u) \right). \notag
\end{align}
Since the inequality is strict and holds for all $u \in X$, we have $\left(\sum\limits_{j\in I} h_j \right) \ne 0$ and
\begin{align}
f(\bar{x}^k)  & < \e  + \min_{u \in X} \left\{ f(u) +  \sum_{i=1}^m \bar{\lambda}_i^k g_{i}(u) \right\}  \notag \\
& =  \e + \varphi(\bar{\lambda}^k).
\end{align}
Second inequality in \eqref{eq:DetPDSolDef} follows from Theorem \ref{Th:MDCompl}.
\qed
\end{proof}

\subsection{Strongly Convex Non-smooth Objective Function}
\label{S:SCNS}
In this subsection, we consider problem \eqref{eq:PrSt} with assumption \eqref{eq:fLipCont} and additional assumption of strong convexity of $f$ and $g$ with the same parameter $\mu$, i.e.,
$$
f(y) \geq f(x) + \la \nabla f(x), y-x \ra + \frac{\mu}{2} \| y-x \|_E^2, \quad x, y \in X
$$
and the same holds for $g$. For example, $f(x) = x^2+|x|$ is a Lipschitz-continuous and strongly convex function on $X = [-1; 1] \subset \mathbb{R}$.  We also slightly modify assumptions on prox-function $d(x)$. Namely, we assume that $0 = \arg \min_{x \in X} d(x)$ and that $d$ is bounded on the unit ball in the chosen norm $\|\cdot\|_E$, that is
\begin{equation}
d(x) \leq \frac{\Omega}{2}, \quad \forall x\in X : \|x \|_E \leq 1,
\label{eq:dUpBound}
\end{equation}
where $\Omega$ is some known number. Finally, we assume that we are given a starting point $x_0 \in X$ and a number $R_0 >0$ such that $\| x_0 - x_* \|_E^2 \leq R_0^2$.

To construct a method for solving problem \eqref{eq:PrSt} under stated assumptions, we use the idea of restarting Algorithm \ref{Alg:MDNS}. The idea of restarting a method for convex problems to obtain faster rate of convergence for strongly convex problems dates back to 1980's, see \cite{nemirovsky1983problem,nesterov1983method}. The algorithm is similar to the one in \cite{bayandina2017adaptive}, but, for the sake of consistency with other parts of the chapter, we use slightly different proof. To show that restarting algorithm is also possible for problems with inequality constraints, we rely on the following lemma.
\begin{lemma}
\label{Lm:SCFuncToArg}
Let $f$ and $g$ be strongly convex functions with the same parameter $\mu$ and $x_*$ be a solution of the problem~\eqref{eq:PrSt}. If, for some $\tilde{x}\in X$,
    $$f(\tilde{x}) - f(x_*) \leq \e, \hspace{0.3cm} g(\tilde{x}) \leq\e,$$
    then
    $$\frac{\mu}{2} \| \tilde{x} - x_* \|_E^2 \leq \e.$$
\end{lemma}
\begin{proof}
\smartqed
Since problem \eqref{eq:PrSt} is regular, by necessary optimality condition \cite{birjukov2010optimization} at the point $x_*$, there exist $\lambda_0,\lambda \geq 0$ not equal to 0 simultaneously, and subgradients $\nabla f(x_*)$, $\nabla g(x_*)$, such that
$$
\la \lambda_0 \nabla f(x_*) + \lambda \nabla g(x_*) , x - x_* \ra \geq 0, \quad \forall x \in X, \quad \lambda g(x_*) = 0.
$$
Since $\lambda_0$ and $\lambda$ are not equal to 0 simultaneously, three cases are possible.

1. $\lambda_0 = 0$ and $\lambda > 0$. Then, by optimality conditions, $g(x_*) = 0$ and $\la \lambda \nabla g(x_*), \tilde{x} - x_*\ra \geq 0$. Thus, by the Lemma assumption and strong convexity,
$$
\e \geq g(\tilde{x}) \geq g(x_*) + \la  \nabla g(x_*), \tilde{x} - x_*\ra + \frac{\mu}{2}\|\tilde{x} - x_*\|_E^2 \geq \frac{\mu}{2}\|\tilde{x} - x_*\|_E^2.
$$

2. $\lambda_0 > 0$ and $\lambda = 0$. Then, by optimality conditions, $\la \lambda_0 \nabla f(x_*), \tilde{x} - x_*\ra \geq 0$. Thus, by the Lemma assumption and strong convexity,
$$
f(x_*) + \e \geq f(\tilde{x}) \geq f(x_*) + \la  \nabla f(x_*), \tilde{x} - x_*\ra + \frac{\mu}{2}\|\tilde{x} - x_*\|_E^2 \geq f(x_*) +  \frac{\mu}{2}\|\tilde{x} - x_*\|_E^2.
$$

3. $\lambda_0 > 0$, $\lambda > 0$. Then, by optimality conditions, $g(x_*) = 0$ and $\la \lambda_0 \nabla f(x_*) +\lambda \nabla g(x_*), \tilde{x} - x_*\ra \geq 0$. Thus, either $\la \nabla g(x_*), \tilde{x} - x_*\ra \geq 0$ and the proof is the same as in the item 1, or $\la  \nabla f(x_*), \tilde{x} - x_*\ra \geq 0$ and the proof is the same as in the item 2.
\qed
\end{proof}

\begin{algorithm}[h!]
\caption{Adaptive Mirror Descent (Non-Smooth Strongly Convex Objective)}
\label{Alg:MDNSSC}
\begin{algorithmic}[1]
  \REQUIRE accuracy $\e > 0$; strong convexity parameter $\mu$; $\Omega$ s.t. $d(x) \leq \frac{\Omega}{2} \quad \forall x\in X : \|x \|_E \leq 1$; starting point $x_0$ and number $R_0$ s.t. $\| x_0 - x_* \|_E^2 \leq R_0^2$.
  \STATE  Set $d_0(x) = d\left(\frac{x-x_0}{R_0}\right)$.
	\STATE Set $p=1$.
	\REPEAT
    \STATE Set $R_p^2 = R_0^2 \cdot 2^{-p}$.
		\STATE Set $\e_p = \frac{\mu R_p^2}{2}$.
		\STATE Set $x_p$ as the output of Algorithm \ref{Alg:MDNS} with accuracy $\e_p$, prox-function $d_{p-1}(\cdot)$ and $\frac{\Omega }{2}$ as $\Theta_0^2$.
		\STATE $d_p(x) \gets d\left(\frac{x - x_p}{R_p}\right)$.
    \STATE Set $p = p + 1$.
  \UNTIL{$p>\log_2 \frac{\mu R_0^2}{2\e}$.}
  \ENSURE $x_p$.
\end{algorithmic}
\end{algorithm}

\begin{theorem}
\label{Th:MDSCCompl}
    Assume that inequalities \eqref{eq:gLipCont} and \eqref{eq:fLipCont} hold and $f$, $g$ are strongly convex with the same parameter $\mu$. Also assume that the prox function $d(x)$ satisfies \eqref{eq:dUpBound} and the starting point $x_0 \in X$ and a number $R_0 >0$ are such that $\| x_0 - x_* \|_E^2 \leq R_0^2$. Then, the point $x_p$ returned by Algorithm \ref{Alg:MDNSSC} is an $\e$-solution to \eqref{eq:PrSt} in the sense of \eqref{eq:DetSolDef} and $\|x_p-x_*\|_E^2\leq \frac{2\e}{\mu}$. At the same time, the total number of iterations of Algorithm \ref{Alg:MDNS} does not exceed
    \begin{equation}
        \left\lceil \log_2 \frac{\mu R_0^2}{2\e}\right\rceil +\frac{32\Omega \max\{M_f^2,M_g^2\} }{\mu \e}.
    \end{equation}
\end{theorem}
\begin{proof}
\smartqed
    Observe that, for all $p \geq 0$, the function $d_{p}(x)$ defined in Algorithm \ref{Alg:MDNSSC} is 1-strongly convex w.r.t. the norm $\|\cdot\|_E / R_{p}$. The conjugate of this norm is $R_{p} \|\cdot\|_{E,*}$. This means that, at each step $k$ of inner Algorithm \ref{Alg:MDNS}, $M_k$ changes to $M_kR_{p-1}$, where $p \geq 1$ is the number of outer iteration.
		
		We show, by induction, that, for all $p\geq 0$, $\|x_p-x_*\|_E^2\leq R_p^2$. For $p=0$ it holds by the assumption on $x_0$ and $R_0$.
		Let us assume that this inequality holds for some $p$ and show that it holds for $p+1$.
		By \eqref{eq:dUpBound}, we have
		$d_{p}(x_*) \leq \frac{\Omega }{2}$. Thus, on the outer iteration $p+1$, by Theorem \ref{Th:MDCompl} and \eqref{eq:DetSolDef}, after at most
		\begin{equation}
		k_{p+1} = \left\lceil\frac{\Omega \max\{M_f^2,M_g^2\} R_{p}^2}{\e_{p+1}^2}\right\rceil
		\label{eq:kpp1Est}
		\end{equation}
		inner iterations, $x_{p+1}=\bar{x}^{k_{p+1}}$ satisfies
		$$
		f(x_{p+1})-f(x_*) \leq \e_{p+1}, \quad g(x_{p+1}) \leq \e_{p+1},
		$$
		where $\e_{p+1} = \frac{\mu R_{p+1}^2}{2}$.
		Then, by Lemma \ref{Lm:SCFuncToArg},
		$$
		\| x_{p+1} - x_* \|_E^2 \leq \frac{2\e_{p+1}}{\mu} = R_{p+1}^2.
		$$
		Thus, we proved that, for all $p\geq 0$, $\|x_p-x_*\|_E^2\leq R_p^2 =  R_0^2 \cdot 2^{-p}$. At the same time, we have, for all $p\geq 1$,
		$$
		f(x_{p})-f(x_*) \leq \frac{\mu R_{0}^2}{2} \cdot 2^{-p}, \quad g(x_{p}) \leq \frac{\mu R_{0}^2}{2} \cdot 2^{-p}.
		$$
		Thus, if $p > \log_2 \frac{\mu R_0^2}{2\e}$, $x_p$ is an $\e$-solution to \eqref{eq:PrSt} in the sense of \eqref{eq:DetSolDef} and
		$$
		\|x_p-x_*\|_E^2\leq R_0^2 \cdot 2^{-p} \leq \frac{2\e}{\mu}.
		$$
		
		Let us now estimate the total number $N$ of inner iterations, i.e. the iterations of Algorithm \ref{Alg:MDNS}. Let us denote $\hat{p} =  \left\lceil \log_2 \frac{\mu R_0^2}{2\e}\right\rceil$. According to \eqref{eq:kpp1Est}, we have
		\begin{align}
		N &= \sum_{p=1}^{\hat{p}} k_p \leq \sum_{p=1}^{\hat{p}} \left(1 + \frac{\Omega \max\{M_f^2,M_g^2\} R_{p}^2}{\e_{p+1}^2}\right) = \sum_{p=1}^{\hat{p}} \left(1 + \frac{16\Omega \max\{M_f^2,M_g^2\}2^p}{\mu^2 R_0^2}\right) \notag \\
		& \leq \hat{p} + \frac{32\Omega \max\{M_f^2,M_g^2\}2^{ \hat{p} }}{\mu^2 R_0^2} \leq \hat{p} + \frac{32\Omega \max\{M_f^2,M_g^2\} }{\mu \e}. \notag
		\end{align}
\qed
\end{proof}

Similarly to Section \ref{S:CNSOF}, let us consider a special type of problem \eqref{eq:PrSt} with strongly convex $g$ given by
\begin{equation}
g(x) = \max\limits_{i \in \{1, ..., m\}} \big\{ g_i(x) \big\}.
\label{eq:gMaxDef1}
\end{equation}
and corresponding dual problem
\begin{equation}
    \varphi (\lambda) = \min\limits_{x\in X} \Big\{ f(x) + \sum\limits_{i=1}^{m} \lambda_i g_i(x) \Big\} \rightarrow \max\limits_{\lambda_i \geq 0, i \in  \{1, ..., m\}} \varphi (\lambda).
\notag
\end{equation}
On each outer iteration $p$ of Algorithm \ref{Alg:MDNSSC}, there is the last inner iteration $k_p$ of Algorithm \ref{Alg:MDNS}. We define approximate dual solution as $\lambda_p = \bar{\lambda}^{k_p}$, where $\bar{\lambda}^{k_p}$ is defined in \eqref{eq:lamDef}. We modify Algorithm \ref{Alg:MDNSSC} to return a pair $(x_p,\lambda_p)$.

Combining Theorem \ref{Th:MDPDCompl} and Theorem \ref{Th:MDSCCompl}, we obtain the following result.
\begin{theorem}
\label{Th:MDPDSCCompl}
    Assume that $g$ is given by \eqref{eq:gMaxDef1}, inequalities \eqref{eq:gLipCont} and \eqref{eq:fLipCont} hold and $f$, $g$ are strongly convex with the same parameter $\mu$. Also assume that the prox function $d(x)$ satisfies \eqref{eq:dUpBound} and the starting point $x_0 \in X$ and a number $R_0 >0$ are such that $\| x_0 - x_* \|_E^2 \leq R_0^2$.
		Then, the pair $(x_p,\lambda_p)$ returned by Algorithm \ref{Alg:MDNSSC} satisfies
		$$
        f(x_p) - \varphi(\lambda_p) \leq \e, \quad g(x_p) \leq \e.
    $$
		and $\|x_p-x_*\|_E^2\leq \frac{2\e}{\mu}$. At the same time, the total number of inner iterations of Algorithm \ref{Alg:MDNS} does not exceed
    $$
        \left\lceil \log_2 \frac{\mu R_0^2}{2\e}\right\rceil + \frac{32\Omega \max\{M_f^2,M_g^2\} }{\mu \e}.
    $$
\end{theorem}

\subsection{General Convex Objective Function}
In this subsection, we assume that the objective function $f$ in \eqref{eq:PrSt} might not satisfy \eqref{eq:fLipCont} and, hence, its subgradients could be unbounded. One of the examples is a quadratic function.
We also assume that inequality \eqref{eq:dx*Bound} holds.

We further develop ideas in \cite{nesterov2004introduction,nesterov2015subgradient} and adapt them for problem \eqref{eq:PrSt}, in a way that our algorithm allows to use non-Euclidean proximal setup, as does Mirror Descent, and does not require to know the constant $M_g$. Following \cite{nesterov2004introduction}, given a function $f$ for each subgradient  $\nabla f(x)$ at a point $y \in X$, we define
\begin{equation}
v_f[y](x)=\left\{
\begin{aligned}
&\left\la\frac{\nabla f(x)}{\|\nabla f(x)\|_{E,*}},x-y\right\ra, \quad &\nabla f(x) \ne 0\\
&0 &\nabla f(x) = 0\\
\end{aligned}
\right.,\quad x \in X.
\label{eq:vfDef}
\end{equation}

\begin{algorithm}[h!]
\caption{Adaptive Mirror Descent (General Convex Objective)}
\label{Alg:MDG}
\begin{algorithmic}[1]
  \REQUIRE accuracy $\e > 0$; $\Theta_0$ s.t. $d(x_{*}) \leq \Theta_0^2$.
  \STATE  $x^0 = \arg\min\limits_{x\in X} d(x)$.
	\STATE Initialize the set $I$ as empty set.
	\STATE Set $k=0$.
	\REPEAT
    \IF{$g(x^k) \leq \e$}
      \STATE $h_k = \frac{\e}{\| \nabla f(x^k) \|_{E,*}}$
      \STATE $x^{k+1} =\mathrm{Mirr}[x^k](h_k \nabla f(x^k))$ ("productive step")
      \STATE Add $k$ to $I$.
    \ELSE
      \STATE $h_k = \frac{\e}{\| \nabla g(x^k) \|_{E,*}^2}$
      \STATE $x^{k+1} = \mathrm{Mirr}[x^k](h_k \nabla g(x^k))$ ("non-productive step")
    \ENDIF
    \STATE Set $k = k + 1$.
  \UNTIL{$|I|+\sum\limits_{j \in J} \frac{1}{\| \nabla g(x^j) \|_{E,*}^2} \geq \frac{2\Theta_0^2}{\e^2} $}
  \ENSURE $\bar{x}^k := \arg \min_{x^j, j\in I} f(x^j)$
\end{algorithmic}
\end{algorithm}

The following result gives complexity estimate for Algorithm \ref{Alg:MDG} in terms of $v_f[x_*](x)$. Below we use this theorem to establish complexity result for smooth objective $f$.
\begin{theorem}
	\label{Th:MDGCompl}
	Assume that inequality \eqref{eq:gLipCont} holds and a known constant $\Theta_0 > 0$ is such that $
d(x_{*}) \leq \Theta_0^2$. Then, Algorithm \ref{Alg:MDG} stops after not more than
	\begin{equation}
	k = \left\lceil\frac{2\max\{1,M_g^2\} \Theta_0^2}{\e^2}\right\rceil
	\label{eq:MDGComplEst}
	\end{equation}
	iterations and it holds that
	$\min_{i \in I} v_f[x_*](x^i) \leq  \e$ and $g(\bar{x}^k)\leq \e$.
\end{theorem}
\begin{proof}
\smartqed
First, let us prove that the inequality in the stopping criterion holds for $k$ defined in \eqref{eq:MDGComplEst}. Denote $[k] = \{ i\in \{0,...,k-1\} \}$, $J = [k] \setminus I$. By \eqref{eq:gLipCont}, we have that, for any $j \in J$, $\| \nabla g(x^j) \|_{E,*} \leq M_g$. Hence, since $|I| + |J| = k$, by \eqref{eq:MDGComplEst}, we obtain
$$
|I|+\sum\limits_{j \in J} \frac{1}{\| \nabla g(x^j) \|_{E,*}^2} \geq |I| + \frac{|J|}{M_g^2} \geq \frac{k}{\max\{1,M_g^2\}} \geq \frac{2\Theta_0^2}{\e^2}.
$$

From Lemma \ref{Lm:MDProp} with $u=x_*$ and $\Delta = 0$, by the definition of $h_i$, $i \in I$, we have, for all $i \in I$,
\begin{align}
\e v_f[x_*](x^i)  &= \e \left\la\frac{\nabla f(x^i)}{\|\nabla f(x^i)\|_{E,*}},x^i-x_*\right\ra = h_i \la \nabla f (x^i), x^i - x_* \ra &  \notag \\
& \leq \frac{h_i^2}{2} \| \nabla f (x^i) \|^2_{E,*} + V[x^i](x_*) - V[x^{i+1}](x_*) \notag \\
&= \frac{\e^2}{2}+ V[x^i](x_*) - V[x^{i+1}](x_*) .
\label{eq:iinIEst}
\end{align}
Similarly, by the definition of $h_i$, $i \in J$, we have, for all $i \in J$,
\begin{align}
\frac{\e (g(x^i)-g(x_*))}{\|\nabla g(x^i)\|_{E,*}^2} &= h_i \big( g(x^i) - g(x_*) \big) \leq \frac{h_i^2}{2} \| \nabla g (x^i) \|^2_{E,*} + V[x^i](x_*) - V[x^{i+1}](x_*) \notag \\
& = \frac{\e^2}{2\| \nabla g (x^i) \|^2_{E,*}}  + V[x^i](x_*) - V[x^{i+1}](x_*). \notag
\end{align}
Whence, using that, for all $i \in J$, $g(x^i)-g(x_*) \geq g(x^i)>\e$, we have
\begin{equation}
- \frac{\e^2}{2\| \nabla g (x^i) \|^2_{E,*}}  + V[x^i](x_*) - V[x^{i+1}](x_*) > 0.
\label{eq:iinJEst}
\end{equation}
Summing up inequalities \eqref{eq:iinIEst} for $i\in I$ and applying \eqref{eq:iinJEst} for $i\in J$, we obtain
\begin{align}
\e |I| \min\limits_{i\in I}v_f[x_*](x^i)  \leq \e \sum\limits_{i\in I}v_f[x_*](x^i)  < \frac{\varepsilon^2}{2} \cdot |I| + \Theta_0^2 - \sum\limits_{i\in J}\frac{\e^2}{2\| \nabla g (x^i) \|^2_{E,*}},
	\notag
	\end{align}
where we also used that, by definition of $x^0$ and \eqref{eq:dx*Bound},
$$
V[x^0](x_{*}) = d(x_{*}) - d(x^0) - \la \nabla d(x^0), x_{*} - x^0 \ra \leq  d(x_{*}) \leq \Theta_0^2.
$$
If the stopping criterion in Algorithm \ref{Alg:MDG} is fulfilled, we get
$$
\e |I| \min\limits_{i\in I}v_f[x_*](x^i)  < \e^2|I|.
$$
Since the inequality is strict, the set $I$ is not empty and the output point $\bar{x}^k$ is correctly defined. Dividing both sides of the last inequality by $\e|I|$, we obtain the first statement of the Theorem. By definition of $\bar{x}^k$, it is obvious that $g(\bar{x}^k) \leq \e$.
\qed
\end{proof}

To obtain the complexity of our algorithm in terms of the values of the objective function $f$, we define non-decreasing function
\begin{equation}
\omega(\tau)=
\left\{
\begin{aligned}
&\max\limits_{x\in X}\{f(x)-f(x_*):\|x-x_*\|_{E}\leq \tau\} \quad &\tau \geq 0,\\
&0 & \tau <0 .\\
\end{aligned}
\right.
\label{eq:omDef}
\end{equation}
and use the following lemma from \cite{nesterov2004introduction}.
\begin{lemma}
\label{Lm:fGrowthOm}
Assume that $f$ is a convex function. Then, for any $x \in X$,
\begin{equation}
f(x) - f(x_*) \leqslant \omega(v_f[x_*](x)).
\label{eq:fGrowth}
\end{equation}
\end{lemma}

\begin{corollary}
\label{Col:MDSCompl}
Assume that the objective function $f$ in \eqref{eq:PrSt} is given as $f(x) = \max_{i \in \{1,...,m\}} f_i(x)$, where $f_i(x)$, $i = 1,...,m$ are differentiable with Lipschitz-continuous gradient
\begin{equation}
\|\nabla f_i(x)-\nabla f_i(y)\|_{E,*} \leq L_i\|x-y\|_{E} \quad \forall x,y\in X, \quad i \in \{1,...,m\}.
\label{eq:fLipSm}
\end{equation}
Then $\bar{x}^k$ is $\widetilde{\varepsilon}$-solution to \eqref{eq:PrSt} in the sense of \eqref{eq:DetSolDef}, where $$\widetilde{\varepsilon} = \max\{\e, \e \max_{i = 1,...,m}\|\nabla f_i(x_*)\|_{E,*}+\e^2\max_{i = 1,...,m}L_i/2\}.$$
\end{corollary}
\begin{proof}
\smartqed
As it was shown in Theorem \ref{Th:MDGCompl}, $g(\bar{x}^k)\leq \e$. It follows from \eqref{eq:fLipSm} that
\begin{align}
f_i(x) & \leq f_i(x_*)+\la \nabla f_i(x_*), x-x_*\ra +\frac{1}{2}L_i||x-x_*||_{E}^2 \notag \\
& \leq f_i(x_*)+\|\nabla f_i(x_*)\|_{E,*}\|x-x_*\|_{E}+\frac{1}{2}L_i||x-x_*||_{E}^2, \quad i=1,...,m. \notag
\end{align}
Whence, $\omega(\tau) \leq \tau \max_{i = 1,...,m}\|\nabla f_i(x_*)\|_{E,*}+\frac{\tau^2\max_{i = 1,...,m}L_i}{2}$.
By Lemma \ref{Lm:fGrowthOm}, non-decreasing property of $\omega$ and Theorem \ref{Th:MDGCompl}, we obtain
\begin{align}
f(\bar{x}^k)-f(x_*) &= \min_{i\in I} f(x^i) - f(x_*) \leq \min_{i\in I}  \omega(v_f[x_*](x^i)) \notag \\
& \leq  \omega(\min_{i\in I} v_f[x_*](x^i)) \leq \omega(\e) \notag \\
&\leq \e \max_{i = 1,...,m}\|\nabla f_i(x_*)\|_{E,*}+\frac{\e^2\max_{i = 1,...,m}L_i}{2}. \notag
\end{align}
\qed
\end{proof}

\section{Randomization for Constrained Problems}
\label{S:Rand}
In this section, we consider randomized version of problem \eqref{eq:PrSt}. This means that we still can use the value of the function $g(x)$ in an algorithm, but, instead of subgradients of $f$ and $g$, we use their stochastic approximations. We combine the idea of switching subgradient method \cite{polyak1967general} and  Stochastic Mirror Descent method introduced in \cite{nemirovski2009robust}. More general case of stochastic optimization problems with expectation constraints is studied in \cite{lan2016algorithms}. We consider convex problems as long as strongly convex  and, for each case, we have two types of algorithms. The first one allows to control expectation of the objective residual $f(\tilde{x})-f(x_*)$ and inequality infeasibility $g(\tilde{x})$, where $\tilde{x}$ is the output of the algorithm. The second one allows to control probability of large deviation for these two quantities.

We introduce the following new assumptions. Given a point $x \in X$, we can calculate stochastic subgradients $\nabla f(x, \xi), \nabla g(x, \zeta)$, where $\xi, \zeta$ are random vectors. These stochastic subgradients satisfy
\begin{equation}
    \Ex\big[ \nabla f(x, \xi) \big]  = \nabla f(x) \in \partial f(x), \quad
    \Ex\big[ \nabla g(x, \zeta)  \big] = \nabla g(x) \in \partial g(x),
   \label{eq:unbiased}
\end{equation}
and
\begin{equation}
\| \nabla f(x, \xi) \|_{E,*} \leq M_f, \hspace{0.3cm} \| \nabla g(x, \zeta) \|_{E,*} \leq M_g, \quad \text{a.s. in } \xi,\zeta.
\label{eq:boundStochSubgr}
\end{equation}

To motivate these assumptions, we consider the following example.
\begin{example}\cite{bayandina2017adaptivest}
Consider Problem \eqref{eq:PrSt} with
$$
    f(x) = \frac{1}{2} \langle Ax, x \rangle,
$$
where $A$ is given $n\times n$ matrix, $X = S(1)$ being standard unit simplex, i.e.\\ $X=\{ x \in \R^n_+: \sum_{i=1}^n x_i = 1\}$, and
$$
	g(x) = \max\limits_{i\in \{1,...,m\}} \big\{ \la c_i, x \ra\big\},
$$
where $\big\{ c_i \big\}_{i=1}^{m}$ are given vectors in~$\mathbb{R}^n$.

Even if the matrix $A$ is sparse, the gradient $\nabla f(x) = Ax$ is usually not. The exact computation of the gradient takes $O(n^2)$ arithmetic operations, which is expensive when $n$ is large. In this setting, it is natural to use randomization to construct a stochastic approximation for $\nabla f(x)$.
Let $\xi$ be a random variable taking its values in $\{1, \dots, n\}$ with probabilities $(x_1, \dots, x_n)$ respectively. Let $A^{\langle i \rangle}$ denote the $i$-th column of the matrix $A$. Since $x\in S_n(1)$,
\begin{gather*}\mathbb{E}\big[ A^{\langle \xi \rangle} \big] = A^{\langle 1 \rangle} \underbrace{\mathbb{P}\big( \xi = 1 \big)}_{x_1} + \dots + A^{\langle n \rangle} \underbrace{\mathbb{P}\big( \xi = n \big)}_{x_n} \\
= A^{\langle 1 \rangle} x_1 + \dots + A^{\langle n \rangle} x_n = Ax.
\end{gather*}
Thus, we can use $A^{\langle \xi \rangle}$ as stochastic subgradient, which can be calculated in $O(n)$ arithmetic operations.
\end{example}

\subsection{Convex Objective Function, Control of Expectation}
\label{S:SMDExp}
In this subsection, we consider convex optimization problem \eqref{eq:PrSt} in randomized setting described above. In this setting the output of the algorithm is random. Thus, we need to change the notion of approximate solution. Let $x_*$ be a solution to \eqref{eq:PrSt}. We say that a (random) point $\tilde{x} \in X$ is an \textit{expected $\e$-solution} to \eqref{eq:PrSt} if
\begin{equation}
    \label{eq:ExpSolDef}
        \Ex f(\tilde{x}) - f(x_{*}) \leq \e, \quad \text{and } \; g(\tilde{x})\;\leq \e \;\;\; \text{a.s.}
\end{equation}

We also introduce a stronger assumption than \eqref{eq:dx*Bound}. Namely, we assume that we know a constant $\Theta_0 > 0$ such that
\begin{equation}
    \sup\limits_{x, y \in X} V[x](y) \leq \Theta_0^2.
\label{eq:VBound}
\end{equation}
The main difference between the method, which we describe below, and the method in \cite{lan2016algorithms} is the adaptivity of our method both in terms of stepsize and stopping rule, which means that we do not need to know the constants $M_f,M_g$ in advance.
We assume that on each iteration of the algorithm independent realizations of $\xi$ and $\zeta$ are generated. The algorithm is similar to the one in  \cite{bayandina2017adaptivest}, but, for the sake of consistency with other parts of the chapter, we use slightly different proof.

\begin{algorithm}[h!]
\caption{Adaptive Stochastic Mirror Descent}
\label{Alg:ASMD}
\begin{algorithmic}[1]
  \REQUIRE accuracy $\e > 0$; $\Theta_0$ s.t. $V[x](y) \leq \Theta_0^2, \quad \forall x,y \in X$.
  \STATE $x^0 = \arg\min\limits_{x\in X} d(x)$.
	\STATE Initialize the set $I$ as empty set.
	\STATE Set $k=0$.
	\REPEAT
    \IF{$g(x^k) \leq \e$.}
      \STATE $M_k = \| \nabla f(x^k,\xi^k) \|_{E,*}$.
      \STATE $h_k = \Theta_0 \Big(\sum\limits_{i=0}^k M_i^2\Big)^{-1/2}$.
      \STATE $x^{k+1} =\mathrm{Mirr}[x^k](h_k \nabla f(x^k, \xi^k))$ ("productive step").
      \STATE Add $k$ to $I$.
    \ELSE
      \STATE $M_k = \| \nabla g(x^k,\zeta^k) \|_{E,*}$.
      \STATE $h_k = \Theta_0 \Big(\sum\limits_{i=0}^k M_i^2\Big)^{-1/2}$.
      \STATE $x^{k+1} = \mathrm{Mirr}[x^k](h_k \nabla g(x^k,\zeta^k))$ ("non-productive step").
    \ENDIF
    \STATE Set $k = k + 1$.
  \UNTIL{$k \geq \frac{2\Theta_0}{\e} \Big(\sum\limits_{i=0}^{k-1} M_i^2\Big)^{1/2} $.}
  \ENSURE $\bar{x}^k = \frac{1}{|I|} \sum\limits_{k\in I} x^k$.
\end{algorithmic}
\end{algorithm}

\begin{theorem}
\label{Th:SMDCompl}
  Let equalities \eqref{eq:unbiased} and inequalities \eqref{eq:boundStochSubgr} hold. Assume that a known constant $\Theta_0 > 0$ is such that $V[x](y) \leq \Theta_0^2, \quad \forall x,y \in X$. Then, Algorithm \ref{Alg:ASMD} stops after not more than
	\begin{equation}
	k = \left\lceil\frac{4\max\{M_f^2,M_g^2\} \Theta_0^2}{\e^2}\right\rceil
	\label{eq:SMDComplEst}
	\end{equation}
	iterations and $\bar{x}^k$ is an expected $\e$-solution to \eqref{eq:PrSt} in the sense of \eqref{eq:ExpSolDef}.
\end{theorem}
\begin{proof}
\smartqed

First, let us prove that the inequality in the stopping criterion holds for $k$ defined in \eqref{eq:SMDComplEst}.
By \eqref{eq:boundStochSubgr}, we have that, for any $i \in \{0,...,k-1\}$, $M_i \leq \max\{M_f,M_g\}$. Hence, by \eqref{eq:SMDComplEst}, $\frac{2\Theta_0}{\e}\left(\sum\limits_{j =0 }^{k-1} M_j^2\right)^{1/2} \leq \frac{2\Theta_0}{\e} \max\{M_f,M_g\} \sqrt{k} \leq k$.

  Denote $[k] = \{ i\in \{0,...,k-1\} \}$, $J = [k] \setminus I$ and
    \begin{equation}
		\label{eq:deltaiDef}
        \delta_i = \begin{cases}
            \langle \nabla f(x^i, \xi^i) - \nabla f(x^i),  x_* - x^i  \rangle, \text{ if } i\in I, \\
            \langle \nabla g(x^i, \zeta^i) - \nabla g(x^i),  x_* - x^i \rangle, \text{ if } i\in J.
        \end{cases}
    \end{equation}
From Lemma \ref{Lm:MDProp} with $u=x_*$ and $\Delta = \nabla f(x^i, \xi^i) - \nabla f(x^i)$, we have, for all $i \in I$,
$$
h_i \big( f(x^i) - f(x_*) \big) \leq \frac{h_i^2}{2} \| \nabla f (x^i,\xi^i) \|^2_{E,*} + V[x^i](x_*) - V[x^{i+1}](x_*) + h_i\delta_i
$$
and, from Lemma \ref{Lm:MDProp} with $u=x_*$ and $\Delta = \nabla g(x^i, \zeta^i) - \nabla g(x^i)$, for all $i \in J$,
$$
h_i \big( g(x^i) - g(x_*) \big) \leq \frac{h_i^2}{2} \| \nabla g (x^i,\zeta^i) \|^2_{E,*} + V[x^i](x_*) - V[x^{i+1}](x_*) + h_i\delta_i.
$$
Dividing each inequality by $h_i$ and summing up these inequalities for $i$ from $0$ to $k-1$, using the definition of $h_i$, $i \in \{0,...,k-1\}$, we obtain
	\begin{align}
  &\sum\limits_{i\in I}  \big( f(x^i) - f(x_{*}) \big) + \sum\limits_{i\in J}  \big( g(x^i) - g(x_{*}) \big) \notag \\
  & \leq \sum\limits_{i\in [k]} \frac{h_i M_i^2}{2}  + \sum\limits_{i\in [k]}\frac{1}{h_i} \big( V[x^i](x_{*}) - V[x^{i+1}](x_{*}) \big) + \sum_{i \in [k]} \delta_i
	\label{eq:ThSMDCproof1}
	\end{align}
  Using \eqref{eq:VBound}, we get
    \begin{align}
        & \sum\limits_{i = 0}^{k-1} \frac{1}{h_i} \big( V[x^i](x_{*}) - V[x^{i+1}](x_{*})  \big) \notag \\
				& = \frac{1}{h_0} V[x^0] (x_*) + \sum\limits_{i=0}^{k-2} \Big( \frac{1}{h_{i+1}} - \frac{1}{h_i} \Big) V[x^{i+1}](x_*) - \frac{1}{h_{k-1}} V[x^{k}](x_*) \notag \\
        & \leq \frac{\Theta_0^2}{h_0} + \Theta_0^2 \sum\limits_{k=0}^{k-2} \Big( \frac{1}{h_{i+1}} - \frac{1}{h_i} \Big) = \frac{\Theta_0^2}{h_{k-1}}. \notag
    \end{align}
		Whence, by the definition of stepsizes $h_i$,
		\begin{align}
		&\sum\limits_{i\in I}  \big( f(x^i) - f(x_{*}) \big) + \sum\limits_{i\in J}  \big( g(x^i) - g(x_{*}) \big) \leq  \sum\limits_{i\in [k]} \frac{h_i M_i^2}{2} +\frac{\Theta_0^2}{h_{k-1}} + \sum_{i \in [k]} \delta_i  \notag \\
		& \leq \sum\limits_{i=0}^{k-1} \frac{\Theta_0}{2} \frac{M_i^2}{\left( \sum_{j=0}^i M_j^2 \right)^{1/2}} + \Theta_0\left( \sum_{i=0}^{k-1} M_i^2 \right)^{1/2} + \sum_{i \in [k]} \delta_i \notag \\
		& \leq 2 \Theta_0 \left( \sum_{i=0}^{k-1} M_i^2 \right)^{1/2} + \sum_{i \in [k]} \delta_i, \notag
\end{align}
    where we used inequality $\sum\limits_{i=0}^{k-1}  \frac{M_i^2}{\left( \sum_{j=0}^i M_j^2 \right)^{1/2}} \leq 2\left( \sum_{i=0}^{k-1} M_i^2 \right)^{1/2}$, which can be proved by induction.
		Since, for $i\in J$, $g(x^i) - g(x_{*}) \geq g(x^i) > \e$, by convexity of $f$, the definition of $\bar{x}^k$, and the stopping criterion, we get
    \begin{align}
		&|I| \big( f(\bar{x}^k) - f(x_*) \big) < \e |I| - \e k + 2\Theta_0 \Big( \sum\limits_{i=0}^{k-1} M_i^2 \Big)^{1/2} + \sum\limits_{i=0}^{k-1} \delta_i \leq \e |I| + \sum\limits_{i=0}^{k-1} \delta_i.
    \end{align}
    Taking the expectation and using \eqref{eq:unbiased}, as long as the inequality is strict and the case of $I = \emptyset$ is impossible, we obtain
    \begin{equation}
        \Ex f(\bar{x}^k) - f(x_*) \leq \e .
    \end{equation}
    At the same time, for $i\in I$ it holds that $g(x^i) \leq \e$. Then, by the definition of $\bar{x}^k$ and the convexity of $g$,
    $$
         g(\bar{x}^k) \leq \frac{1}{|I|}\sum\limits_{i\in I} g(x^i) \leq \e .
    $$
\qed
\end{proof}

\subsection{Convex Objective Function, Control of Large Deviation}
\label{S:SMDHP}
In this subsection, we consider the same setting as in previous subsection, but change the notion of approximate solution. Let $x_*$ be a solution to \eqref{eq:PrSt}. Given $\e > 0$ and $\sigma \in (0,1)$, we say that a point $\tilde{x} \in X$ is an \textit{$(\e,\sigma)$-solution} to \eqref{eq:PrSt} if
\begin{equation}
    \label{eq:HPrSolDef}
        \Prb\left\{ f(\tilde{x}) - f(x_{*}) \leq \e, \quad g(\tilde{x}) \leq \e \right\} \geq 1-\sigma.
\end{equation}
As in the previous subsection, we use an assumption expressed by inequality \eqref{eq:VBound}.
We assume additionally to \eqref{eq:boundStochSubgr} that inequalities \eqref{eq:gLipCont} and \eqref{eq:fLipCont} hold. Unfortunately, it is not clear, how to obtain large deviation guarantee for an adaptive method. Thus, in this section, we assume that the constants $M_f$, $M_g$ are known and use a simplified algorithm. We assume that on each iteration of the algorithm independent realizations of $\xi$ and $\zeta$ are generated.

\begin{algorithm}[h!]
\caption{Stochastic Mirror Descent}
\label{Alg:SMD}
\begin{algorithmic}[1]
  \REQUIRE accuracy $\e > 0$; maximum number of iterations $N$; $M_f$, $M_g$ s.t. \eqref{eq:gLipCont}, \eqref{eq:fLipCont}, \eqref{eq:boundStochSubgr} hold.
  \STATE $x^0 = \arg\min\limits_{x\in X} d(x)$.
	\STATE Set $h = \frac{\e}{\max\{M_f^2,M_g^2\}}$.
	\STATE Set $k=0$.
	\REPEAT
    \IF{$g(x^k) \leq \e$.}
      \STATE $x^{k+1} =\mathrm{Mirr}[x^k](h \nabla f(x^k, \xi^k))$ ("productive step").
      \STATE Add $k$ to $I$.
    \ELSE
      \STATE $x^{k+1} = \mathrm{Mirr}[x^k](h \nabla g(x^k,\zeta^k))$ ("non-productive step").
    \ENDIF
    \STATE Set $k = k + 1$.
  \UNTIL{$k \geq N$.}
  \ENSURE If $I \ne \emptyset$, then $\bar{x}^k = \frac{1}{|I|} \sum\limits_{k\in I} x^k$. Otherwise $\bar{x}^k = NULL$.
\end{algorithmic}
\end{algorithm}

To analyze Algorithm \ref{Alg:SMD} in terms of large deviation bound, we need the following known result, see, e.g. \cite{boucheron2013concentration}.
\begin{lemma}[Azuma-Hoeffding Inequality]\label{Lm:Azuma}
    Let $\eta^1, \dots, \eta^n$ be a sequence of independent random variables taking their values in some set $\Xi$, and let $Z = \phi (\eta^1, \dots, \eta^n)$ for some function $\phi : \Xi^n \rightarrow \R$. Suppose that a.~s.
    $$
        \big| \Ex[ Z | \eta^1, \dots, \eta^i ] - \Ex[ Z | \eta^1, \dots, \eta^{i-1} ] \big| \leq c_i, \hspace{0.3cm} i = 1, \dots, n,
    $$
		where $c_i$, $i \in \{1,...,n\}$ are deterministic.
    Then, for each $t \geq 0$
    $$
        \Prb \big( Z - \Ex Z \geq t \big) \leq \exp\Bigg\{ -\frac{t^2}{2\sum\limits_{i=1}^n c_i^2} \Bigg\}.
    $$
\end{lemma}

\begin{theorem}
\label{Th:SMDHPCompl}
	Let equalities \eqref{eq:unbiased} and inequalities \eqref{eq:gLipCont}, \eqref{eq:fLipCont}, \eqref{eq:boundStochSubgr} hold. Assume that a known constant $\Theta_0 > 0$ is such that $V[x](y) \leq \Theta_0^2$, $\forall x,y \in X$, and the confidence level satisfies $\sigma \in (0,0.5)$. Then, if in Algorithm~\ref{Alg:SMD}

\begin{equation}
    N =  \Bigg\lceil 70\frac{\max\{M_f^2,M_g^2\} \Theta_0^2}{\e^2}  \ln\frac{1}{\sigma}  \Bigg\rceil    ,
    \label{eq:SMDHPComplEst}
    \end{equation}
$\bar{x}^k$ is an $(\e,\sigma)$-solution to \eqref{eq:PrSt} in the sense of \eqref{eq:HPrSolDef}.
\end{theorem}

\begin{proof}
	\smartqed
	Let us denote $M = \max\{M_f,M_g\}$. In the same way as we obtained \eqref{eq:ThSMDCproof1} in the proof of Theorem \ref{Th:SMDCompl}, we obtain
	\begin{align}
  &h\sum\limits_{i\in I}  \big( f(x^i) - f(x_{*}) \big) + h\sum\limits_{i\in J}  \big( g(x^i) - g(x_{*}) \big) \notag \\
  & \leq \frac{h^2 M^2 k}{2}  + V[x^0](x_{*}) + h \sum_{i =0}^{k-1} \delta_i ,
	\notag
	\end{align}
	where $\delta_i$, $i=0,...,k-1$ are defined in \eqref{eq:deltaiDef}.
			Since, for $i\in J$, $g(x^i) - g(x_{*}) \geq g(x^i) > \e$, by convexity of $f$, the definition of $\bar{x}^k$ and $h$, we get
    \begin{align}
		&h|I| \big( f(\bar{x}^k) - f(x_*) \big) < \e h|I| - \frac{\e^2 k}{2M^2} + \Theta_0^2 + h\sum\limits_{i=0}^{k-1} \delta_i .\label{eq:ThSMDHPComplProof1}
    \end{align}		
    Using Cauchy-Schwarz inequality, \eqref{eq:gLipCont}, \eqref{eq:fLipCont}, \eqref{eq:boundStochSubgr}, \eqref{eq:VBound}, we have
    \begin{align}
		h\big| \delta_i \big| & \leq 2hM \| x^i - x^* \| \notag \\
		& \leq 2hM \sqrt{2 V[x^i](x^*)} \leq 2\sqrt{2} hM\Theta_0 = 2\sqrt{2} \frac{\e \Theta_0}{M}. \notag
    \end{align}
		Now we use Lemma \ref{Lm:Azuma} with $Z = \sum\limits_{i=0}^{k-1} h\delta_i$.
    Clearly, $\Ex Z = \Ex \Big[ \sum\limits_{i=0}^{k-1} h\delta_i \Big] = 0$ and we can take $c_i = 2\sqrt{2} \frac{\e \Theta_0}{M}$. Then, by Lemma \ref{Lm:Azuma}, for each $t \geq 0$,
    $$
        \Prb \left\{ \sum\limits_{i=0}^{k-1} h\delta_i \geq t \right\} \leq \exp\left( - \frac{t^2}{2\sum\limits_{i=0}^{k-1} c_i^2} \right) = \exp\left( - \frac{t^2 M^2}{16\e^2 \Theta_0^2k} \right).
    $$
		In other words, for each $\sigma \in (0,1)$
    \begin{equation}
        \Prb  \left\{  \sum\limits_{i=0}^{k-1} h\delta_i \geq \frac{4\e\Theta_0}{M} \sqrt{k \ln \Big( \frac{1}{\sigma} \Big)} \right\} \leq \sigma. \notag
    \end{equation}
    Applying this inequality to \eqref{eq:ThSMDHPComplProof1}, we obtain, for any $\sigma \in (0,1)$,
    $$
        \Prb  \left\{ h |I| \big( f(\bar{x}^k) - f(x_*) \big) < \e h|I| - \frac{\e^2 k}{2M^2} + \Theta_0^2 + \frac{4\e\Theta_0}{M} \sqrt{k \ln \Big( \frac{1}{\sigma} \Big)} \right\} \geq 1-\sigma.
    $$
    Then, by \eqref{eq:SMDHPComplEst} , we have
    \begin{align}
    - \frac{\e^2 k}{2M^2} + \Theta_0^2 + \frac{4\e\Theta_0}{M} \sqrt{k \ln \Big( \frac{1}{\sigma} \Big)} &< \Theta_0^2 \left(-\frac{71}{2} \ln \Big( \frac{1}{\sigma} \Big) + 1 + 4\ln \Big( \frac{1}{\sigma} \Big) \sqrt{71} \right) \notag \\
		& < \Theta_0^2 \left(-\frac{3}{2}\ln \Big( \frac{1}{\sigma} \Big) +1 \right).
		\end{align}
		Since $\sigma \leq 0.5 < \exp(-2/3)$, we have $-\frac{3}{2}\ln \Big( \frac{1}{\sigma} \Big) +1 <0$ and
    $$
        \Prb  \left\{ h  |I|  \big( f(\bar{x}^k) - f(x^*) \big) < h |I| \e \right\} \geq 1-\sigma.
    $$
    Thus, with probability at least $1-\sigma$, the inequality is strict, the case of $I = \emptyset$ is impossible, and $\bar{x}^k$ is correctly defined. Dividing the both sides of it by $h \cdot |I|$, we obtain that  $\Prb  \left\{ f(\bar{x}^k) - f(x^*) \leq \e \right\} \geq 1-\sigma$.
    At the same time, for $i\in I$ it holds that $g(x^i) \leq \e$. Then, by the definition of $\bar{x}^k$ and the convexity of $g$, again with probability at least $1-\sigma$
    $$
        g(\bar{x}^k) \leq \frac{1}{|I| } \sum\limits_{i\in I} g(x^i) \leq \e .
    $$	
		Thus, $\bar{x}^k$ is an $(\e,\sigma)$-solution to \eqref{eq:PrSt} in the sense of \eqref{eq:HPrSolDef}.
\qed
\end{proof}

\subsection{Strongly Convex Objective Function, Control of Expectation}
In this subsection, we consider the setting of Subsection \ref{S:SMDExp}, but, as in Subsection \ref{S:SCNS}, make the following additional assumptions. First, we assume that functions $f$ and $g$ are strongly convex. Second, without loss of generality, we assume that $0 = \arg \min_{x \in X} d(x)$. Third, we assume that we are given a starting point $x_0 \in X$ and a number $R_0 >0$ such that $\| x_0 - x_* \|_E^2 \leq R_0^2$. Finally, we make the following assumption (cf. \eqref{eq:dUpBound}) that $d$ is bounded in the following sense. Assume that $x_*$ is some fixed point and $x$ is a random point such that $\Ex_x \big[ \| x-x_* \|_E^2 \big] \leq R^2$, then
\begin{equation}
    \Ex_x \Big[ d\Big( \frac{x-x_*}{R} \Big) \Big] \leq \frac{\Omega}{2},
		\label{eq:expdUpBound}
\end{equation}
where $\Omega$ is some known number and $\Ex_x$ denotes the expectation with respect to random vector $x$. For example, this assumption holds for Euclidean proximal setup.
Unlike the method introduced in \cite{lan2016algorithms} for strongly convex problems, we present a method, which is based on the restart of Algorithm \ref{Alg:SMD}. Unfortunately, it is not clear, whether the restart technique can be combined with adaptivity to constants $M_f$, $M_g$. Thus, we assume that these constants are known.

\begin{algorithm}[h!]
\caption{Stochastic Mirror Descent (Strongly Convex Objective, Expectation Control)}
\label{Alg:SMDSC}
\begin{algorithmic}[1]
  \REQUIRE accuracy $\e > 0$; strong convexity parameter $\mu$; $\Omega$ s.t. $\Ex_x \Big[ d\Big( \frac{x-x_*}{R} \Big) \Big] \leq \frac{\Omega}{2}$ if\\ $\Ex_x \big[ \| x-x_* \|_E^2 \big] \leq R^2$; starting point $x_0$ and number $R_0$ s.t. $\| x_0 - x_* \|_E^2 \leq R_0^2$.
  \STATE  Set $d_0(x) = d\left(\frac{x-x_0}{R_0}\right)$. 
	\STATE Set $p=1$.
	\REPEAT
    \STATE Set $R_p^2 = R_0^2 \cdot 2^{-p}$.
		\STATE Set $\e_p = \frac{\mu R_p^2}{2}$.
		\STATE Set $N_p =  \Bigg\lceil \frac{\max\{M_f^2,M_g^2\} \Omega R_{p-1}^2}{\e_p^2} \Bigg\rceil$
		\STATE Set $x_p$ as the output of Algorithm \ref{Alg:SMD} with accuracy $\e_p$, number of iterations $N_p$, prox-function $d_{p-1}(\cdot)$ and $\frac{\Omega}{2}$ as $\Theta_0^2$.
		\STATE $d_p(x) \gets d\left(\frac{x - x_p}{R_p}\right)$.
    \STATE Set $p = p + 1$.
  \UNTIL{$p>\log_2 \frac{\mu R_0^2}{2\e}$.}
  \ENSURE $x_p$.
\end{algorithmic}
\end{algorithm}

The following lemma can be proved in the same way as Lemma \ref{Lm:SCFuncToArg}.
\begin{lemma}
\label{Lm:SCExpFuncToArg}
Let $f$ and $g$ be strongly convex functions with the same parameter $\mu$ and $x_*$ be a solution of problem~\eqref{eq:PrSt}. Assume that, for some random $\tilde{x}\in X$,
    $$\Ex f(\tilde{x}) - f(x_*) \leq \e, \hspace{0.3cm} g(\tilde{x}) \leq\e.$$
    Then
    $$\frac{\mu}{2} \Ex \| \tilde{x} - x_* \|_E^2 \leq \e.$$
\end{lemma}
\begin{theorem}
\label{Th:SMDSCCompl}
	Let equalities \eqref{eq:unbiased} and inequalities \eqref{eq:boundStochSubgr} hold and $f$, $g$ be strongly convex with the same parameter $\mu$.
	Also assume that the prox function $d(x)$ satisfies \eqref{eq:expdUpBound} and the starting point $x_0 \in X$ and a number $R_0 >0$ are such that $\| x_0 - x_* \|_E^2 \leq R_0^2$.	
	Then, the point $x_p$ returned by Algorithm \ref{Alg:SMDSC} is an expected $\e$-solution to \eqref{eq:PrSt} in the sense of \eqref{eq:ExpSolDef} and  $\Ex \|x_p-x_*\|_E^2\leq \frac{2\e}{\mu}$. At the same time, the total number of inner iterations of Algorithm \ref{Alg:SMD} does not exceed
    \begin{equation}
        \left\lceil \log_2 \frac{\mu R_0^2}{2\e}\right\rceil + \frac{32\Omega \max\{M_f^2,M_g^2\} }{\mu \e}.
    \end{equation}
	\end{theorem}
\begin{proof}
\smartqed
		Let us denote $M = \max\{M_f,M_g\}$.
    Observe that, for all $p \geq 0$, the function $d_{p}(x)$ defined in Algorithm \ref{Alg:SMDSC} is 1-strongly convex w.r.t. the norm $\|\cdot\|_E / R_{p}$. The conjugate of this norm is $R_{p} \|\cdot\|_{E,*}$. This means that, at each outer iteration $p$, $M$ changes to $MR_{p-1}$, where $p$ is the number of outer iteration.
		We show by induction that, for all $p\geq 0$, $\Ex\|x_p-x_*\|_E^2\leq R_p^2$. For $p=0$ it holds by the definition of $x_0$ and $R_0$.
		
		Let us assume that this inequality holds for some $p-1$ and show that it holds for $p$.
		At iteration $p$, we start Algorithm \ref{Alg:SMD} with starting point $x_{p-1}$ and stepsize $h_p = \frac{\e_p}{M^2R_{p-1}^2}$.
		Using the same steps as in the proof of Theorem \ref{Th:SMDHPCompl}, after $N_p$ iterations of Algorithm \ref{Alg:SMD} (see \eqref{eq:ThSMDHPComplProof1}), we obtain
		\begin{align}
		&h_p|I_p| \big( f(\bar{x}_p^k) - f(x_*) \big) < \e_p h_p|I_p| - \frac{\e_p^2 N_p}{2M^2R_{p-1}^2} + V_{p-1}[x_{p-1}](x_*) + h_p\sum\limits_{i=0}^{N_p-1} \delta_i ,\label{eq:ThSMDSCComplProof1}
    \end{align}	
		where $V_{p-1}[z](x)$ is the Bregman divergence corresponding to $d_{p-1}(x)$ and $I_p$ is the set of "productive steps". Using the definition of $d_{p-1}$, we have
		$$
		V_{p-1}[x_{p-1}](x_*) = d_{p-1}(x_{*}) - d_{p-1}(x_{p-1}) - \la \nabla d_{p-1}(x_{p-1}), x_{*} - x_{p-1} \ra \leq  d_{p-1}(x_{*}).
		$$
		Taking expectation with respect to $x_{p-1}$ in \eqref{eq:ThSMDSCComplProof1} and using inductive assumption $\Ex\|x_{p-1}-x_*\|_E^2\leq R_{p-1}^2$ and \eqref{eq:expdUpBound}, we obtain, substituting $N_p$,
		\begin{align}
		&h_p|I_p| \big( f(\bar{x}_p^k) - f(x_*) \big) < \e_p h_p|I_p| - \frac{\e_p^2 N_p}{2M^2R_{p-1}^2} + \frac{\Omega}{2} + h_p\sum\limits_{i=0}^{N_p-1} \delta_i \leq \e_p h_p|I_p| + h_p\sum\limits_{i=0}^{N_p-1} \delta_i,\label{eq:ThSMDSCComplProof2}
    \end{align}
		Taking the expectation and using \eqref{eq:unbiased}, as long as the inequality is strict and the case of $I_p = \emptyset$ is impossible, we obtain
    \begin{equation}
        \Ex f(\bar{x}_p^k) - f(x_*) \leq \e_p .
    \end{equation}
    At the same time, for $i \in I_p$ it holds that $g(x^i) \leq \e_p$. Then, by the definition of $\bar{x}_p^k$ and the convexity of $g$,
    $$
         g(\bar{x}_p^k) \leq \frac{1}{|I_p|}\sum\limits_{i\in I_p} g(x^i) \leq \e_p.
    $$
		Thus, we can apply Lemma \ref{Lm:SCExpFuncToArg} and obtain
		$$
		\Ex \| x_{p} - x_* \|_E^2 \leq \frac{2\e_{p}}{\mu} = R_{p}^2.
		$$
		Thus, we proved that, for all $p\geq 0$, $\Ex \|x_p-x_*\|_E^2\leq R_p^2 =  R_0^2 \cdot 2^{-p}$. At the same time, we have, for all $p\geq 1$,
		$$
		\Ex f(x_{p})-f(x_*) \leq \frac{\mu R_{0}^2}{2} \cdot 2^{-p}, \quad g(x_{p}) \leq \frac{\mu R_{0}^2}{2} \cdot 2^{-p}.
		$$
		Thus, if $p > \log_2 \frac{\mu R_0^2}{2\e}$, $x_p$ is an $\e$-solution to \eqref{eq:PrSt} in the sense of \eqref{eq:ExpSolDef} and
		$$
		\Ex \|x_p-x_*\|_E^2\leq R_0^2 \cdot 2^{-p} \leq \frac{2\e}{\mu}.
		$$
		
		Let us now estimate the total number $N$ of inner iterations, i.e. the iterations of Algorithm \ref{Alg:MDNS}. Let us denote $\hat{p} =  \left\lceil \log_2 \frac{\mu R_0^2}{2\e}\right\rceil$. We have
		\begin{align}
		N &= \sum_{p=1}^{\hat{p}} N_p \leq \sum_{p=1}^{\hat{p}} \left(1 + \frac{\Omega \max\{M_f^2,M_g^2\} R_{p-1}^2}{\e_{p}^2}\right)
		= \sum_{p=1}^{\hat{p}} \left(1 + \frac{16\Omega \max\{M_f^2,M_g^2\}2^p}{\mu^2 R_0^2}\right) \notag \\
		& \leq \hat{p} + \frac{32\Omega \max\{M_f^2,M_g^2\}2^{ \hat{p} }}{\mu^2 R_0^2} \leq \hat{p} + \frac{32\Omega \max\{M_f^2,M_g^2\} }{\mu \e}. \notag
		\end{align}
\qed
\end{proof}

\subsection{Strongly Convex Objective Function, Control of Large Deviation}
In this subsection, we consider the setting of Subsection \ref{S:SMDHP}, but make the following additional assumptions. First, we assume that functions $f$ and $g$ are strongly convex. Second, without loss of generality, we assume that $0 = \arg \min_{x \in X} d(x)$. Third, we assume that we are given a starting point $x_0 \in X$ and a number $R_0 >0$ such that $\| x_0 - x_* \|_E^2 \leq R_0^2$. Finally, instead of \eqref{eq:VBound}, we assume that the Bregman divergence satisfies quadratic growth condition
\begin{equation}
    V[z](x) \leq \frac{\Omega}{2}\|x-z\|_E^2, \quad x,z \in X.
		\label{eq:VQuadGrowth}
\end{equation}
where $\Omega$ is some known number. For example, this assumption holds for Euclidean proximal setup.
Unlike the method introduced in \cite{lan2016algorithms} for strongly convex problems, we present a method, which is based on the restart of Algorithm \ref{Alg:SMD}. Unfortunately, it is not clear, whether the restart technique can be combined with adaptivity to constants $M_f$, $M_g$. Thus, we assume that these constants are known.

\begin{algorithm}[h!]
\caption{Stochastic Mirror Descent (Strongly Convex Objective, Control of Large Deviation)}
\label{Alg:SMDSCHP}
\begin{algorithmic}[1]
  \REQUIRE accuracy $\e > 0$; strong convexity parameter $\mu$; $\Omega$ s.t. $V[x](y) \leq \frac{\Omega}{2}\|x-y\|_E^2, \quad x,y \in X$; starting point $x_0$ and number $R_0$ s.t. $\| x_0 - x_* \|_E^2 \leq R_0^2$.
  \STATE  Set $d_0(x) = d\left(\frac{x-x_0}{R_0}\right)$.
	\STATE Set $p=1$.
	\REPEAT
    \STATE Set $R_p^2 = R_0^2 \cdot 2^{-p}$.
		\STATE Set $\e_p = \frac{\mu R_p^2}{2}$.
		\STATE Set $N_p =  \Bigg\lceil 70 \frac{\max\{M_f^2,M_g^2\} \Omega R_{p-1}^2}{\e_p^2}  \ln\left(\frac{1}{\sigma}\log_2 \frac{\mu R_0^2}{2\e} \right) \Bigg\rceil$.
		\STATE Set $X_p = \{x \in X: \|x-x_{p-1}\|_E^2 \leq R_{p-1}^2 \}$.
		\STATE Set $x_p$ as the output of Algorithm \ref{Alg:SMD} with accuracy $\e_p$, number of iteration $N_p$, prox-function $d_{p-1}(\cdot)$, $\Omega$ as $\Theta_0^2$ and $X_p$ as the feasible set.
		\STATE $d_p(x) \gets d\left(\frac{x - x_p}{R_p}\right)$.
    \STATE Set $p = p + 1$.
  \UNTIL{$p>\log_2 \frac{\mu R_0^2}{2\e}$.}
  \ENSURE $x_p$.
\end{algorithmic}
\end{algorithm}

\begin{theorem}
\label{Th:SMDSCHPCompl}
Let equalities \eqref{eq:unbiased} and inequalities \eqref{eq:gLipCont}, \eqref{eq:fLipCont}, \eqref{eq:boundStochSubgr} hold.
Let $f$, $g$ be strongly convex with the same parameter $\mu$.
Also assume that the Bregman divergence $V[z](x)$ satisfies \eqref{eq:VQuadGrowth} and the starting point $x_0 \in X$ and a number $R_0 >0$ are such that $\| x_0 - x_* \|_E^2 \leq R_0^2$.	
Then, the point $x_p$ returned by Algorithm \ref{Alg:SMDSCHP} is an $(\e,\sigma)$-solution to \eqref{eq:PrSt} in the sense of \eqref{eq:HPrSolDef} and $\|x_p-x_*\|_E^2\leq \frac{2\e}{\mu}$ with probability at least $1-\sigma$. At the same time, the total number of inner iterations of Algorithm \ref{Alg:SMD} does not exceed
    \begin{equation}
        \left\lceil \log_2 \frac{\mu R_0^2}{2\e}\right\rceil +  \frac{2240\Omega \max\{M_f^2,M_g^2\}  }{\mu \e}\left(\ln\frac{1}{\sigma} + \ln \log_2 \frac{\mu R_0^2}{2\e} \right). \notag
    \end{equation}
\end{theorem}
\begin{proof}
\smartqed
		Let us denote $M = \max\{M_f,M_g\}$.
    Observe that, for all $p \geq 0$, the function $d_{p}(x)$ defined in Algorithm \ref{Alg:SMDSCHP} is 1-strongly convex w.r.t. the norm $\|\cdot\|_E / R_{p}$. The conjugate of this norm is $R_{p} \|\cdot\|_{E,*}$. This means that, at each outer iteration $p$, $M$ changes to $MR_{p-1}$, where $p$ is the number of outer iteration.
		
		Let $A_p$, $p\geq 0$ be the event $A_p = \{\|x_p - x_*\|_E^2 \leq R_p^2 \}$ and $\bar{A}_p$ be its complement. Note that, by the definition of $x_0$ and $R_0$, $A_0$ holds with probability 1. Denote $\hat{p} =  \left\lceil \log_2 \frac{\mu R_0^2}{2\e}\right\rceil$.
		
		We now show by induction that, for all $p\geq 1$, $\Prb \{A_p|A_{p-1} \} \geq 1- \frac{\sigma}{\hat{p}}$. 		
		By inductive assumption, $A_{p-1}$ holds and we have $\|x_{p-1} - x_*\|_E^2 \leq R_{p-1}^2$. At iteration $p$, we start Algorithm \ref{Alg:SMD} with starting point $x_{p-1}$, feasible set $X_{p}$ and Bregman divergence $V_{p-1}[z](x)$ corresponding to $d_{p-1}(x)$. Thus, by \eqref{eq:VQuadGrowth}, we have
		\begin{align}
		\max_{x,z \in X_p}  V_{p-1}[z](x) & = \max_{x,z \in X_p}  d\left(\frac{x-x_{p-1}}{R_{p-1}}\right) - d\left(\frac{z-x_{p-1}}{R_{p-1}}\right) \notag \\
		& \hspace{1em}- \left\la \nabla d\left(\frac{z-x_{p-1}}{R_{p-1}}\right), \frac{x-x_{p-1}}{R_{p-1}} - \frac{z-x_{p-1}}{R_{p-1}} \right\ra \notag \\
		& = \max_{x,z \in X_p}  V\left[\frac{z-x_{p-1}}{R_{p-1}}\right] \left(\frac{x-x_{p-1}}{R_{p-1}}\right)  \notag \\
		& \leq \max_{x,z \in X_p}  \frac{\Omega \|x-z\|_E^2}{2R_{p-1}^2} \leq \Omega. \notag
		\end{align}	
		Hence, by Theorem \ref{Th:SMDHPCompl} with $\sigma_p = \frac{\sigma}{\hat{p}}$, after $N_p$ iterations of Algorithm \ref{Alg:SMD}, we have
		$$
		\Prb\left\{ f(x_p) - f(x_{*}) \leq \e_p, \quad g(x_p) \leq \e_p | A_{p-1} \right\} \geq 1-\frac{\sigma}{\hat{p}}.
		$$
		Whence, by Lemma \ref{Lm:SCFuncToArg},
		$$
		\Prb\left\{ A_p | A_{p-1} \right\}  = \Prb\left\{ \|x_{p} - x_*\|_E^2 \leq R_{p}^2 | A_{p-1} \right\}  \geq 1-\frac{\sigma}{\hat{p}},
		$$
		which finishes the induction proof.
		
		At the same time,
		\begin{align}
		& \Prb\left\{ f(x_{\hat{p}}) - f(x_{*}) > \e_{\hat{p}} \quad \text{or} \quad g(x_{\hat{p}}) > \e_{\hat{p}}\right\} \notag \\
		& = \Prb \left\{\left. f(x_{\hat{p}}) - f(x_{*}) > \e_{\hat{p}} \quad \text{or} \quad g(x_{\hat{p}}) > \e_{\hat{p}} \right| A_{{\hat{p}}-1} \cup \bar{A}_{{\hat{p}}-1} \right\} \notag \\
		& = \Prb \left\{\left. f(x_{\hat{p}}) - f(x_{*}) > \e_{\hat{p}} \quad \text{or} \quad g(x_{\hat{p}}) > \e_{\hat{p}}\right| A_{{\hat{p}}-1}  \right\} \Prb \{ A_{{\hat{p}}-1} \} \notag \\
		& + \Prb \left\{\left. f(x_{\hat{p}}) - f(x_{*}) > \e_{\hat{p}} \quad \text{or} \quad g(x_{\hat{p}}) > \e_{\hat{p}} \right| \bar{A}_{{\hat{p}}-1}  \right\} \Prb \{\bar{A}_{{\hat{p}}-1} \}  \notag \\
		& \leq \frac{\sigma}{\hat{p}} + \Prb \{\bar{A}_{{\hat{p}}-1} \} \stackrel{(\ast)}{\leq}  \frac{\sigma}{\hat{p}}+ \Prb \left\{f(x_{{\hat{p}}-1}) - f(x_{*}) > \e_{{\hat{p}}-1} \quad \text{or} \quad g(x_{{\hat{p}}-1}) > \e_{{\hat{p}}-1}\right\}\notag\\
		&  \leq 2\cdot\frac{\sigma}{\hat{p}} + \Prb \{\bar{A}_{{\hat{p}}-2} \} \leq ... \leq \frac{\hat{p} - 1}{\hat{p}} \cdot \sigma + \Prb \{\bar{A}_{1}\},
        \end{align}
		where $(\ast)$ follows from Lemma \ref{Lm:SCFuncToArg}.
		Using that $\Prb\{A_1\} = \Prb\{A_1 | A_0\} \geq 1- \frac{\sigma}{\hat{p}}$ and, hence, $\Prb \{\bar{A}_{1}\} \leq \frac{\sigma}{\hat{p}}$, we obtain
		\begin{align}
		&\Prb\left\{ f(x_{\hat{p}}) - f(x_{*}) \leq \e, \quad g(x_{\hat{p}}) \leq \e \right\} \geq 1-\sigma . \notag
		\end{align}
		Hence,
		$$
		\Prb\left\{ \|x_{\hat{p}}-x_*\|_E^2 \leq \frac{2\e}{\mu} \right\} \geq 1-\sigma.
		$$
		
		Let us now estimate the total number $N$ of inner iterations, i.e. the iterations of Algorithm \ref{Alg:SMD}. We have
		\begin{align}
		N &= \sum_{p=1}^{\hat{p}} N_p \leq \sum_{p=1}^{\hat{p}} \left(1 + 70\frac{\Omega \max\{M_f^2,M_g^2\} R_{p-1}^2}{\e_{p}^2}\ln\left(\frac{1}{\sigma}\log_2 \frac{\mu R_0^2}{2\e} \right)\right) \notag \\
		&= \sum_{p=1}^{\hat{p}} \left(1 + 1120\frac{\Omega \max\{M_f^2,M_g^2\}2^p}{\mu^2 R_0^2}\ln\left(\frac{1}{\sigma}\log_2 \frac{\mu R_0^2}{2\e} \right)\right) \notag \\
		& \leq \hat{p} + 2240\frac{\Omega \max\{M_f^2,M_g^2\}2^{ \hat{p} }}{\mu^2 R_0^2}\ln\left(\frac{1}{\sigma}\log_2 \frac{\mu R_0^2}{2\e} \right) \notag \\
		& \leq \hat{p} + 2240\frac{\Omega \max\{M_f^2,M_g^2\} }{\mu \e}\left(\ln\frac{1}{\sigma} + \ln\log_2 \frac{\mu R_0^2}{2\e} \right). \notag
		\end{align}
\qed
\end{proof}

\section{Discussion}
We conclude with several remarks concerning possible extensions of the described results.

Obtained results can be easily extended for \textit{composite optimization problems} of the form
\begin{align}
\label{eq:CompPrSt}
    \min \{ f(x) + c(x) : x \in X \subset E, g(x) + c(x) \leq 0\},
\end{align}
where $X$ is a convex closed subset of finite-dimensional real vector space $E$, $f: X \to \R$, $g: E \to \R$, $c: X \to \R$ are convex functions. Mirror Descent for unconstrained composite problems was proposed in  \cite{duchi2010composite}, see also \cite{xiao2010dual} for corresponding version of Dual Averaging \cite{nesterov2009primal-dual}. To deal with composite problems \eqref{eq:CompPrSt}, the Mirror Descent step should be changed to
$$
x_+ = \mathrm{Mirr}[x](p) = \arg\min\limits_{u\in X} \big\{ \la p, u \ra + d(u)+c(u) - \la \nabla d(x) , u \ra \big\} \quad \forall x \in X^0,
$$
where $X^0$ is defined in Section \ref{S:MD}. The counterpart of Lemma \ref{Lm:MDProp} is as follows.
\begin{lemma}
	\label{Lm:CMDProp}
    Let $f$ be some convex function over a convex closed set $X$, $h > 0$ be a stepsize, $x \in X^0$. Let the point $x_+$ be defined by
    $ x_+ = \mathrm{Mirr}[x](h \cdot (\nabla f(x) + \Delta)) $, where $\Delta \in E^*$. Then, for any $u \in X$,
    \begin{align}
    h \cdot \big( f(x) - f(u) + c(x_+) & - c(u) + \la \Delta , x - u \ra \big) \notag \\
		& \leq h \cdot \la \nabla f(x) + \Delta , x - u \ra - h \cdot \la \nabla c(x_+), u - x_+ \ra \notag \\
		& \leq \frac{h^2}{2} \| \nabla f (x)  + \Delta \|^2_{E,*} + V[x](u) - V[x_+](u).		\notag		
    \end{align}
\end{lemma}

We considered restarting Mirror Descent only in the case of strongly convex functions. A possible extension can be in applying the restart technique to the case of uniformly convex functions $f$ and $g$ introduced in \cite{polyak1967existence} and satisfying
$$
f(y) \geq f(x) + \la \nabla f(x), y-x \ra + \frac{\mu}{2} \| y-x \|_E^\rho, \quad x, y \in X,
$$
where $\rho \geq 2$, and the same holds for $g$.
Restarting Dual Averaging \cite{nesterov2009primal-dual} to obtain subgradient methods for minimizing such functions without functional constraints, both in deterministic and stochastic setting, was suggested in \cite{juditsky2014deterministic}.
Another option is, as it was done in \cite{roulet2017sharpness} for deterministic unconstrained problems, to use sharpness condition of $f$ and $g$
$$
\mu \left(\min_{x_* \in X_*} \|x-x_*\|_E \right)^{\rho} \leq f(x) -f_*, \quad \forall x \in X,
$$
where  $f_*$ is the minimum value of $f$, $X_*$ is the set of minimizers of $f$ in Problem \eqref{eq:PrSt}, and the same holds for $g$.

In stochastic setting, motivated by randomization for deterministic problems, we considered only problems with available values of $g$. As it was done in \cite{lan2016algorithms}, one can consider more general problems of minimizing an expectation of a function under inequality constraint given by $\Ex G(x,\eta) \leq 0$, where $\eta$ is random vector. In this setting one can deal only with stochastic approximation of this inequality constraint.

\begin{acknowledgement}
The authors are very grateful to Anatoli Juditsky, Arkadi Nemirovski and Yurii Nesterov for fruitful discussions. The research by P. Dvurechensky and A. Gasnikov presented in Section 4 was conducted in IITP RAS and supported by the Russian Science Foundation grant (project 14-50-00150). The research by F. Stonyakin presented in subsection 3.3 was partially supported by the grant of the President of the Russian Federation for young candidates of sciences, project no. MK-176.2017.1.

\end{acknowledgement}

\bibliographystyle{abbrv}
\bibliography{references}

\end{document}